\documentclass[12pt]{amsart}
\usepackage{latexsym}
\usepackage{amsfonts}
\usepackage{amsmath}
\usepackage{amsthm}
\usepackage{amssymb}
\usepackage{amscd}

\usepackage[english]{babel}
\usepackage{enumerate}
\newcommand{\hot}{\hat{\otimes}}

\newcommand{\RR}{\mathbb{R}}

\newcommand{\gothg}{\mathfrak{g}}
\newcommand{\g}{\gothg}
\newcommand{\fk}{\mathfrak{k}}

\newcommand{\cM}{\mathcal{M}}
\newcommand{\Fre}{{Fr\'{e}chet \,}}
\newcommand{\Hom}{\operatorname{Hom}}
\newcommand{\C}{\mathbb{C}}

\newcommand{\into}{\hookrightarrow}

\newtheorem{thm}{Theorem}[section]
\newtheorem{thm*}{Theorem}
\newtheorem{lem}[thm]{Lemma}
\newtheorem{lem*}[thm*]{Lemma}
\newtheorem{prop}[thm]{Proposition}
\newtheorem{cor}[thm]{Corollary}
\newtheorem{prop*}[thm*]{Proposition}
\newtheorem{cor*}[thm*]{Corollary}

\newtheorem{defn}[thm]{Definition}
\newtheorem{defn*}[thm*]{Definition}
\newtheorem{thm-defn}[thm]{Theorem-Definition}

\theoremstyle{remark}

\begin{document}

\author{Dmitry Gourevitch}
\address{Faculty of Mathematics and Computer Science, Weizmann
Institute of Science, POB 26, Rehovot 76100, Israel }
\email{dimagur@weizmann.ac.il}
\author{Alexander Kemarsky}
\address{Mathematics Department, Technion - Israel Institute of Technology, Haifa, 32000 Israel}
\email{alexkem@tx.technion.ac.il}

\title{Irreducible representations of product of real reductive groups}

\begin{abstract}
Let $G_1,G_2$ be reductive Lie groups and $(\pi,V)$ a smooth, irreducible, admissible representation
of $G_1 \times G_2$. We prove that $(\pi,V)$ is the completed tensor product of
$(\pi_i,V_i)$, $i=1,2$, where $(\pi_i,V_i)$ is a smooth,irreducible,admissible representation
of $G_i$, $i=1,2$. We deduce this from the analogous theorem for Harish-Chandra modules, for which one direction was proven in \cite[Appendix A]{AG} and the other direction we prove here.

As a corollary, we deduce that strong Gelfand property for a pair $H \subset G$ of real reductive groups is equivalent to the usual Gelfand property of the pair $\Delta H \subset G \times H$.

\end{abstract}

\date{\today}

\maketitle

\begin{section}{Introduction}
Let $G_1,G_2$ be reductive Lie groups , $\gothg_i $ be the Lie algebra of $G_i$. Fix $K_i$ - a maximal compact subgroup
of $G_i$ ($i=1,2$).
Let $\cM(\g_i,K_i)$ be the category of admissible Harish-Chandra $(\g_i,K_i)$-modules  and $\cM(G_i)$ be the category of smooth admissible \Fre representations of moderate growth.
We also denote by $Irr(G_i)$ and $Irr(\g_i,K_i)$ the isomorphism classes or irreducible objects in the above categories.

In this note we prove

\begin{thm}\label{thm:gK}
Let $M\in Irr(\g_1\times \g_2,K_1\times K_2)$. Then
there exist $M_i \in Irr(\g_i,K_i)$ such that $M=M_1\otimes M_2$.
\end{thm}

The converse statement, saying that for  irreducible $M_i \in \cM(\g_i,K_i), \, M_1\otimes M_2$ is irreducible is \cite[Proposition A.0.6]{AG}. By the Casselman-Wallach equivalence of categories $\cM(\g,K) \simeq \cM(G)$, these two statements imply

\begin{thm}\label{thm:G}
A representation $(\pi,V)\in \cM(G_1\times G_2)$ is irreducible if and only if there exist irreducible $(\pi_i,V_i)\in \cM(G_i)$ such that
$(\pi,V) \simeq (\pi_1,V_1) \hot (\pi_2,V_2).$
\end{thm}

Finally, we deduce a consequence of this theorem concerning Gelfand pairs. A pair $(G,H)$ of reductive groups is called a \emph{Gelfand pair} if $H\subset G$ is a closed subgroup and the space $(\pi^*)^H$ of $H$-invariant continuous functionals on any $\pi\in Irr(G)$ has dimension zero or one. It is called a \emph{strong Gelfand pair} or a \emph{multiplicity-free pair} if $\dim\Hom_H(\pi|_H,\tau)\leq 1$ for any $\pi \in Irr(G), \, \tau\in Irr(H)$.

\begin{cor}\label{cor:Gel}
Let $H\subset G$ be reductive groups and let $\Delta H \subset G \times H$ denote the diagonal. Then $(G,H)$ is a multiplicity-free pair if and only if $(G \times H,\Delta H)$ is a Gelfand pair.
\end{cor}

An analog of Corollary \ref{cor:Gel} was proven in \cite{vD} for generalized Gelfand property of arbitrary Lie groups, with smooth representations replaced by smooth vectors in unitary representations.

An analog of Theorem \ref{thm:G} for p-adic groups was proven in \cite[\S\S 2.16]{BZ}. For
a more detailed exposition see \cite[\S\S 10.5]{GoldHun}.

\subsection{Acknowledgements}
We thank Avraham Aizenbud for useful remarks.

\end{section}
\begin{section}{Preliminaries}\label{sec:Prel}

\subsection{Harish-Chandra modules and smooth representations}\label{subsec:HC}

In this subsection we fix a real reductive group $G$ and a maximal compact subgroup $K \subset G$. Let $\g,\fk$ denote the complexified Lie algebras of $G,K$.

\begin{defn}
A $(\g,K)$-module is a  $\g$-module $\pi$ with a locally finite action of $K$ such the two induced actions of $\fk$ coincide and
$\pi(ad(k)(X))=\pi(k)\pi(X)\pi(k^{-1})$ for any $k\in K$ and $X \in \g$.

A finitely-generated $(\g,K)$-module is called admissible if any representation of $K$ appears in it with finite (or zero) multiplicity. In this case we also call it a Harish-Chandra module.
\end{defn}

\begin{lem}[\cite{Wal1}, \S\S 4.2] \label{lem:AdmFinLen}
Any admissible $(\g,K)$-module $\pi$ has  finite length.
\end{lem}

\begin{thm}[Casselman-Wallach, see \cite{Wal2}, \S\S\S 11.6.8] \label{thm:CW}
The functor of taking $K$-finite vectors $HC:\cM(G) \to \cM(\g,K)$  is an equivalence of categories.
\end{thm}

In fact, Casselman and Wallach construct an inverse functor $\Gamma: \cM(\g,K) \to \cM(G)$, that is called Casselman-Wallach globalization functor (see \cite[Chapter 11]{Wal2} or \cite{CasGlob} or, for a different approach, \cite{BerKr}).

\begin{cor} \label{cor:CW}
$ $
\begin{enumerate}[(i)]
\item  \label{it:Ab} The category $\cM(G)$ is abelian.
\item \label{it:ClosIm} Any morphism in $\cM(G)$ has closed image.
\end{enumerate}
\end{cor}
\begin{proof}
(\ref{it:Ab})  $\cM(\g,K)$ is clearly  abelian and by the theorem  is equivalent to $\cM(G)$.

(\ref{it:ClosIm}) Let $\phi:\pi \to \tau$ be a morphism in $\cM(G)$.
Let $\tau' = \overline{Im {\phi}}$, $\pi' = {\pi/\ker {\phi}}$ and $\phi':\pi' \to \tau'$ be the natural morphism. Clearly $\phi'$ is monomorphic and epimorphic in the category $\cM_{}(G)$. Thus by (\ref{it:Ab}) it is an isomorphism. On the other hand, $Im \phi' = Im \phi \subset \overline{Im \phi} = \tau'$. Thus $Im {\phi} = \overline{Im \phi}$.
\end{proof}

%

We will also use the embedding theorem of Casselman.

\begin{thm}
Any irreducible $(\g,K)$-module can be imbedded into a
$(\g,K)$-module of principal series.
\end{thm}

Those two statements have the following corollary.

\begin{cor}
The underlying topological vector space of any admissible smooth \Fre representation is a nuclear \Fre space.
\end{cor}

\begin{defn}
Let $G_1$ and $G_2$ be real reductive groups. Let $(\pi_i,V_i)\in \cM(G_i)$ be admissible smooth \Fre representations of $G_i$. We define $\pi_1 \otimes \pi_2$ to be the natural
representation of $G_1\times G_2$ on the space $V_1 \widehat{\otimes}
V_2$.
\end{defn}

\begin{prop}[\cite{AG}, Proposition A.0.6] \label{prop:irr}
Let $G_{1}$ and $G_{2}$ be real reductive groups. Let $\pi_i \in Irr(\g_i,K_i)$
irreducible admissible Harish-Chandra modules of $G_i$. Then $\pi_1 \otimes \pi_2 \in Irr(\g_1\times \g_2,K_1\times K_2)$.
\end{prop}

We will use the classical statement on irreducible representations of compact groups.
\begin{lem}\label{lem:comp}
Let $K_1,K_2$ be compact groups. A representation $\tau$ of $K_1\times K_2$ is irreducible if and only if there exist irreducible representations $\tau_i$ of $K_i$ such that $\tau \simeq \tau_1\otimes \tau_2$. Note that $\tau_i$ are finite-dimensional, and $\otimes$ is the usual tensor product.
\end{lem}

\begin{cor}\label{cor:HCTen}
Let $G_1$ and $G_2$ be real reductive groups and $(\pi_i,V_i)\in \cM(G_i)$.
Then we have a natural isomorphism $(\pi_1\otimes\pi_2)^{HC}\simeq   \pi_1^{HC} \otimes   \pi_2^{HC}$.
\end{cor}

\end{section}
\begin{section}{Proof of Theorem \ref{thm:gK}}\label{sec:PfThmgK}
Throughout the section $\rho_i$ always denote irreducible representations of $K_1$, $\sigma_j$ always denote   irreducible representations of $K_2$. For a representation $V$ of $K_1$ (or of $K_2$) we will denote by $V^{\rho}$ (resp. by $V^{\sigma}$) the corresponding isotypic component. \\
Let $K:=K_1\times K_2$ and $\g:=\g_1\times \g_2$.

Let $(\pi,V)$ be an irreducible admissible $(\gothg,K)$ - module. We show that there exist non-zero irreducible and admissible
$(\gothg_1,K_1)$-module  $V_1$ and $(\gothg_2,K_2)$-module $V_2$ and a non-zero morphism $V_1 \bigotimes V_2 \to V$. From the irreducibility of $V$ and $V_1 \bigotimes V_2$, we obtain
that $V  \simeq V_1 \bigotimes V_2 $.


Let's first find the module $V_1$. Choose $\tau \in Irr(K)$ such that the isotypic component $V^{\tau}$ is non-zero. By Lemma \ref{lem:comp} $\tau \simeq \rho \otimes \sigma$ for some $\rho\in Irr(K_1), \, \sigma \in Irr(K_2)$.  Let
$W$ be the $(\gothg_1,K_1)$-module generated by $V^{\tau}$.
Note that since the actions of $(\gothg_1,K_1)$ and $(\gothg_2,K_2)$ commute, $W$ is also a
$K_2$-module and $W = W^{\sigma}$.
We claim that $W$ is an admissible $(\gothg_1,K_1)$-module.
Indeed, let $\rho_1$ be an irreducible representation of $K_1$. Then
$W^{\rho_1}  \subseteq V^{\rho_1 \otimes \sigma}$ and as a corollary
$$\text{dim}(W^{\rho_1}) \leq \text{dim}\left( V^{\rho_1 \otimes \sigma} \right) < \infty , $$
since $V$ is an admissible $(\gothg,K)$-module.
%

Now by Lemma \ref{lem:AdmFinLen} $W$ has finite length and thus
there is an irreducible admissible $(\gothg_1,K_1)$-submodule $V_1 \subseteq W$. Thus, we finished the
first stage of the proof. \newline
Let $$
W_2' := \Hom_{(\g_1,K_1)}(V_1,V). $$
Clearly, $W_2' \neq 0$.
Since actions of $(\gothg_1,K_1)$ and $(\gothg_2,K_2)$ on $V$ commute,
$W_2'$ has a natural structure of $(\gothg_2,K_2)$-module.
Take any non-zero morphism $L \in W_2'$ and let $W_2 \subset W_2'$ be the $(\gothg_2,K_2)$-module generated by $L$.

Let us show that $W_2$ is admissible. Choose $\sigma_2 \in Irr(K_2)$.
Let $\rho_2\in Irr(K_1)$  such that
$V_1^{\rho_2} \neq 0 $. Then $V_1^{\rho_2}$ generates $V_1$ and thus for any $L',L'' \in W_2^{\sigma_{2}}$ if $L'$ agrees with $L''$ on $V_1^{\rho_2}$ then $L'=L''$. This gives a linear embedding from $ W_2^{\sigma_{2}}$ into the finite-dimensional space $\Hom_{\C}(V_1^{\rho_2},V^{\rho_2 \otimes \sigma_{2}})$.
Thus $W_2$ is an admissible $(\gothg_2,K_2)$-module.

%

Thus $W_2$ has finite length and therefore there is an irreducible admissible submodule $V_2 \subseteq W_2$.
Define a linear map $\phi:V_1 \bigotimes V_2 \to V $ by the formula $$\phi(v \otimes l) := l(v) $$
 on the pure tensors.
Clearly, this is a non-zero $(\gothg,K)$-map. From the irreducibility of $V$  and of
$V_1 \bigotimes V_2$ (see Proposition \ref{prop:irr}), the result
$$V \simeq V_1 \bigotimes V_2$$ follows.
\end{section}

\begin{section}{Proof of Theorem \ref{thm:G} and Corollary \ref{cor:Gel}}\label{sec:PfCorGel}

\begin{proof}[Proof of Theorem \ref{thm:G}]
First take $\pi_i\in Irr(G_i),$ for $i=1,2.$ Then $\pi_i^{HC}\in Irr(\g_i,K_i)$ and by Proposition \ref{prop:irr} $\pi_1^{HC}\otimes \pi_2^{HC}\in Irr(\g_1\times \g_2,K_1\times K_2)$. By Corollary \ref{cor:HCTen} $(\pi_1\otimes \pi_2)^{HC}\simeq \pi_1^{HC}\otimes \pi_2^{HC}\in Irr(\g_1\times \g_2,K_1\times K_2)$. This implies $\pi_1\otimes \pi_2 \in Irr(G_1\times G_2)$.

Now take $\pi \in Irr(G_1\times G_2)$. Then $\pi^{HC} \in Irr(\g_1\times \g_2,K_1\times K_2)$ and by Theorem \ref{thm:gK} there exist  $(M_i)\in Irr(\g_i,K_i)$ such that $\pi^{HC} \simeq M_1\otimes M_2$. By Theorem \ref{thm:CW} there exist $\pi_i \in Irr(G_i)$ such that $\pi_i^{HC} \simeq M_i$. Then $\pi^{HC}\simeq \pi_1^{HC}\otimes \pi_2^{HC}\simeq (\pi_1\otimes \pi_2)^{HC}$ and by Theorem \ref{thm:CW} this implies $\pi \simeq \pi_1 \otimes \pi_2$.
\end{proof}

Corollary \ref{cor:Gel} follows from Theorem \ref{thm:G} and the following lemma.

\begin{lem}
Let $H \subset G$ be real reductive groups. Let $(\pi,E)$ and
$(\tau,W)$ be admissible smooth \Fre representations of $G$ and
$H$ respectively. Then $Hom_{H}(\pi,\tau)$ is canonically
isomorphic to $Hom_{\Delta H}(\pi \otimes \widetilde{\tau},\C),$ where $\widetilde{\tau}$ denotes the contragredient representation.
\end{lem}

\begin{proof}
For a nuclear \Fre space $V$ we denote by $V'$ its dual space
equipped with the strong topology. Let $\widetilde{W}\subset W'$ denote the
underlying space of $\widetilde{\tau}$. By the theory of nuclear
\Fre spaces (\cite[Chapter 50]{Treves}, we know $Hom_{\C}(E,W) \cong E' \widehat{\otimes} W$
and $Hom_{\C}(E \widehat{\otimes} \widetilde{W}, \C) \cong E'
\widehat{\otimes} \widetilde{W}'$. Thus we have canonical embeddings
$$ Hom_{H}(\pi,\tau) \into Hom_{\Delta H}(\pi \otimes \widetilde{\tau},\C)
\into Hom_{H}(\pi , \widetilde{\tau}')$$
Since the image of any $H$-equivariant map from $\pi$ to $\widetilde{\tau}'$ lies in the space of smooth vectors $\widetilde{\widetilde{\tau}}$, which is canonically isomorphic to $\tau$, the lemma follows.
\end{proof}

%

\end{section}

\end{document}